\documentclass{amsart}

\usepackage{amssymb, amsmath, amsthm, hyperref}

\usepackage[english]{babel}

\author{Rafael Torres}

\title[Einstein metrics and smooth structures on $3\mathbb{CP}^2 \# k\overline{\mathbb{CP}^2}$]{On Einstein metrics, normalized Ricci flow and smooth structures on $3\mathbb{CP}^2 \# k\overline{\mathbb{CP}^2}$.}

\address{California Institute of Technology - Mathematics\\ 1200 E California Blvd\\91125\\Pasadena, CA}

\email{rtorresr@caltech.edu}

\date{September 2, 2010. We thank D. Kotschick for pointing out some useful references.}



\theoremstyle{plain}

\newtheorem{theorem}[equation]{Theorem}

\newtheorem{corollary}[equation]{Corollary}

\newtheorem{proposition}[equation]{Proposition}

\newtheorem{lemma}[equation]{Lemma}

\newtheorem{remark}{Remark}

\theoremstyle{definition}

\newcommand{\Q}{\mathbb{Q}}

\begin{document}

\maketitle

In this paper, first we consider the existence and non-existence of Einstein metrics on the topological 4-manifolds $3\mathbb{CP}^2 \# k \overline{\mathbb{CP}^2}$ (for $k \in \{11, 13, 14, 15, 16, 17, 18\}$) by using the idea of \cite{[RS]} and the constructions in \cite{[PPS]}. Then, we study the existence or non-existence of non-singular solutions of the normalized Ricci flow on the exotic smooth structures of these topological manifolds by employing the obstruction developed in \cite{[MI]}.


\section{Introduction}

Recent years have witnessed a drastic increase in our understanding of the topology and geometry of 4-manifolds and complex surfaces. The newest developments can be exemplified by the construction of simply connected surfaces of general type with small topology \cite{[PPS2], [PPS]}, by the unveiling of a myriad of exotic smooth structures on small 4-manifolds \cite{[AP]}, and by how these manifolds have provided an adequate environment for the study of fundamental questions in Riemannian geometry that were previously out of reach.\\

In particular, intriguing questions regarding Einstein metrics (\cite{[RS]}), and the relation between smooth and geometric structures (like the Yamabe invariant and the normalized Ricci flow) on a given topological 4-manifold (\cite{[MI], [IRS]}) have been immediate beneficiaries of the novel constructions. In this paper we employ the procedure of R. R{\u a}sdeaconu and I. {\c S}uvaina (\cite{[RS]}) to the constructions of H. Park, J. Park and D. Shin (\cite{[PPS]}), and to those of A. Akhmedov and B.D. Park (\cite{[AP]}) to study the (non)-existence of Einstein metrics, and the (non)-existence of non-singular solutions to the normalized Ricci flow on small manifolds (although bigger than those considered in \cite{[RS]}).\\

Our main results are the following.

\begin{theorem} Let $k\in \{11,13, 14, 15, 16, 17, 18\}$. Each of the topological 4-manifolds
\begin{center}
$3\mathbb{CP}^2 \# k \overline{\mathbb{CP}^2}$
\end{center}
admits a smooth structure that has an Einstein metric of scalar curvature $s < 0$, and infinitely many non-diffeomorphic smooth
structures that do not admit Einstein metrics.
\end{theorem}

Regarding the non-singular solutions to the normalized Ricci flow on the exotic smooth structures of the manifolds from Theorem 1 and in the spirit of \cite{[IRS]}, the following result is proven.

\begin{proposition} The topological 4-manifold $M:= 3\mathbb{CP}^2 \# k \overline{\mathbb{CP}^2}$ satisfies the following properties
\begin{enumerate}
\item $M$ admits a smooth structure of negative Yamabe invariant on which there exist non-singular solutions to the normalized Ricci flow.
\item $M$ admits infinitely many smooth structures, all of which have negative Yamabe invariant, and on which there are no non-singular solutions to the normalized Ricci flow for any initial metric.
\end{enumerate}

\end{proposition}

We are also able to prove that for $k\geq 9$, each of the reducible manifolds of Theorem 1 have infinitely many smooth structures that do not carry an Einstein metric, all of which have negative Yamabe invariant, and on which the only solutions to the normalized Ricci flow for any initial metric are all singular (Proposition 11). Moreover, for $k\geq 8$, the manifolds of our theorem does not admit anti-self-dual Einstein metrics (Lemma 12). Theorem 1 and Proposition 2 extend the results in \cite{[RS]} and \cite{[IRS]}, and improve results of \cite{[LI]} and \cite{[BK]}.\\

The (non)-existence of Einstein metrics on different smooth structures on small blow ups of $3\mathbb{CP}^2$ was previously considered by V. Braungart and D. Kotschick in \cite{[BK]}. In that paper, the authors proved instances $k = 17$ and $k = 18$ of Theorem 1. By a result of F. Catanese \cite{[FC]}, in the case $k = 18$, the manifolds with K\"ahler-Einstein metrics in this paper and in \cite{[BK]} are diffeomorphic.\\

The paper is organized as follows. In Section 2 we determine the homeomorphism types of the complex surfaces built by H. Park, J. Park and D. Shin; the second part of the section provides a description of their surfaces. The Third Section contains the construction of an Einstein metric on each of these surfaces of general type. The non-existence of these metrics on the topological prototypes is addressed in Section 4. The proof of Theorem 1 is spread through out the first four sections. In Section 5, we study the sign of the Yamabe invariant and the solutions to the normalized Ricci flow on the exotic smooth structures. That is, Proposition 2 is proven in the fifth and last section.


\section{Homeomorphism type}

The following theorem was proven in \cite{[PPS]}.

\begin{theorem} (H. Park- J.Park - D. Shin). There exist simply connected minimal surfaces of general type with $p_g = 1$, $q = 0$ and $K^2 = 1, 2, 3, 4, 5,6, 8$.
\end{theorem}

Here $K$ denotes the canonical divisor class of the complex surface. Our enterprise starts by pinning down a homeomorphism type for each of these complex surfaces. From now on, let $S$ be one of such surfaces. Surfaces of general type are K\"ahler (see, for example, \cite[Lemma 2]{[LB1]}). Thus, one has
\begin{center}
$b_2^+(S) = 2 p_g + 1 = 3$.\\
\end{center}

On the other hand, we have

\begin{lemma} $b_2^-(S) = 19 - c_1^2(S)$.
\end{lemma}

\begin{proof} The Thom-Hirzebruch Index Theorem (\cite[Theorem I 3.1, p.22]{[BHPV]}) states
\begin{center}
$\sigma(S) = \frac{1}{3}(c_1^2(S) - 2 c_2(S)) = \frac{1}{3}(c_1^2(S) - 2 e(S))$.
\end{center}

The claim follows by substituting $\sigma(S) = b_2^+ - b_2^- = 3 - b_2^-$ and $e(S) = b_2^+ + b_2^- + 2 = 5 + b_2^-$.
\end{proof}

From these computations we also observe

\begin{corollary} These manifolds satisfy the Hitchin-Thorpe inequality \cite[Theorem 1]{[H]}.\\
\end{corollary}

\begin{proof} The claim is $2\chi + 3\sigma >0$. Indeed, we have
\begin{center}
$2\chi + 3\sigma = 19 - b_2^-$,
\end{center}

which is always a positive number for the manifolds considered in this paper.
\end{proof}

It follows from Rokhlin's Theorem \cite{[R]} that the manifolds built in \cite{[PPS]} are non-spin. We are now ready to conclude on the topological prototypes of the minimal surfaces of general type in question by using Freedman's Theorem \cite{[F]}, and Donaldson's results \cite{[D]}. The possible homeomorphism types are arranged in the following proposition.

\begin{proposition} Let $S$ be a simply connected surface of general type with $p_g = 1$, and $q = 0$. The homeomorphism type of $S$ is given as follows
\begin{itemize}

\item If $K^2_S = 1: S \cong_{C^0} 3\mathbb{CP}^2 \# 18\overline{\mathbb{CP}^2}$.
\item If $K^2_S = 2: S \cong_{C^0} 3\mathbb{CP}^2 \# 17\overline{\mathbb{CP}^2}$.
\item If $K^2_S = 3: S \cong_{C^0} 3\mathbb{CP}^2 \# 16\overline{\mathbb{CP}^2}$.
\item If $K^2_S = 4: S \cong_{C^0} 3\mathbb{CP}^2 \# 15\overline{\mathbb{CP}^2}$.
\item If $K^2_S = 5: S \cong_{C^0} 3\mathbb{CP}^2 \# 14\overline{\mathbb{CP}^2}$.
\item If $K^2_S = 6: S \cong_{C^0} 3\mathbb{CP}^2 \# 13\overline{\mathbb{CP}^2}$.
\item If $K^2_S = 8: S \cong_{C^0} 3\mathbb{CP}^2 \# 11\overline{\mathbb{CP}^2}$.
\end{itemize}
\end{proposition}

\begin{remark} In particular notice that the 4-manifolds of H. Park, J. Park and D. Shin (Theorem 3) are exotic symplectic copies of the reducible manifolds of Proposition 6.
\end{remark}

\subsection{Description of the minimal surfaces of general type built by H. Park, J. Park and D. Shin}

The minimal complex surfaces of Theorem 3 have a very similar nature. We proceed to give an sketch of the construction for the example with $K^2 = 6$. The reader is referred to the quoted papers for details.\\

The starting manifold is a particular rational elliptic surface $E(1)$, which is obtained out of blowing-up a well-chosen pencil of cubics in $\mathbb{CP}^2$. Take the double cover of this rational elliptic surface $E(1)$, and call it $Y$. The complex surface $Y$ is an elliptic $K3$ surface; this complex manifold $Y$ is a common material in all of the minimal surfaces produced in \cite{[PPS]}.\\

In particular for the surface with $K^2 = 6$ we are describing in this section, one considers (within $Y$) two $I_8$-singular fibers, two $I_2$-singular fibers, one nodal singular fiber, and three sections.\\

By blowing-up $Y$ 18 times at rightly selected points (see \cite[Section 4.5, Fig. 13]{[PPS]}), one obtains a surface $Z:= Y \# 18\overline{\mathbb{CP}^2}$. The surface $Z$ contains five disjoint linear chains of $\mathbb{CP}^1$'s including the proper transforms of the sections. The linear chains are denoted by the following dual graphs, which have been labeled for the purposes in Section 3:\\

\begin{center}
$\displaystyle\circ_{G_1}^{-2} - \displaystyle\circ_{G_2}^{-2} - \displaystyle\circ_{G_3}^{-3} - \displaystyle\circ_{G_4}^{-9} - \displaystyle\circ_{G_5}^{-2} - \displaystyle\circ_{G_6}^{-2} - \displaystyle\circ_{G_7}^{-2} - \displaystyle\circ_{G_8}^{-2} - \displaystyle\circ_{G_9}^{-3} - \displaystyle\circ_{G_{10}}^{-4}$,\\

$\displaystyle\circ_{H_1}^{-2} - \displaystyle\circ_{H_2}^{-3} - \displaystyle\circ_{H_3}^{-7} - \displaystyle\circ_{H_4}^{-2} - \displaystyle\circ_{H_5}^{-2} - \displaystyle\circ_{H_6}^{-3} - \displaystyle\circ_{H_7}^{-3}$,\\

$\displaystyle\circ_{I_1}^{-7} - \displaystyle\circ_{I_2}^{-2} - \displaystyle\circ_{I_3}^{-2} - \displaystyle\circ_{I_4}^{-2}$,\\

$\displaystyle\circ_{J_1}^{-4} - \displaystyle\circ_{J_2}^{-3} - \displaystyle\circ_{J_3}^{-2}$, \\

$\displaystyle\circ_{L}^{-4}$
\end{center}

One proceeds to contract these five chains of $\mathbb{CP}^1$'s from $Z$. Since Artin's criteria is satisfied (\cite{[AR]}), the contraction produces a projective surface with special quotient singularities. Denote it by $X$. At this step, H. Park, J. Park and D. Shin use $\Q$-Gorenstein smoothings to deal with the singularities. Each singularity admits a local $\Q$-Gorenstein smoothing. In Section 3 of \cite{[PPS]}, they prove that the local smoothings can actually be glued to a global $\Q$-smoothing of the entire singular surface by proving there is no obstruction to do so. The surface of general type $S$ with $p_g = 1$, $q = 0$ and $K^2 = 6$ is a general fiber of the smoothing of $X$; in the papers of H. Park, J. Park and D. Shin, $S$ is denoted by $X_t$.\\

The argument regarding the minimality of $S$ goes as follows. Let $f:Z\rightarrow X$ be the contraction map of the chains of $\mathbb{CP}^1$'s from $Z$ to the singular surface $X$. By using the technique in, for example, Section 5 in \cite{[PPS2]}, one sees that the pullback $f^*K_X$ of the canonical divisor $K_X$ of $X$ is effective and nef. Therefore, $K_X$ is nef as well, which implies the minimality of $S$.

\section{Existence of Einstein metrics}

The existence of an Einstein metric on a certain manifold is hard to prove. In the case of interest of this paper, where the manifold is a minimal complex surface of general type that does not contain any (-2)-curves, the following criterion was found independently by T. Aubin and by S.T. Yau.

\begin{theorem} (Aubin  \cite{[Au]}, Yau \cite{[Ya]}). A compact complex manifold $(M^4, J)$ admits a compatible K\" ahler-Einstein metric with $s < 0$ if and only if its canonical
line bundle $K_M$ is ample. When such a metric exists, it is unique , up to an overall multiplicative constant.
\end{theorem}

In order to apply Theorem 7, the following result needs to be proven.

\begin{proposition} There exist simply connected surfaces of general type with $p_g = 1$, $q = 0$, $K^2 = 1, 2, 3, 4, 5, 6$ or $8$, and ample canonical bundle.
\end{proposition}

The rest of the section is devoted to such endeavor. We carry out the argument for the surface with $K^2 = 6$. The other examples can be dealt with in a similar fashion.

\subsection{Proof of Proposition 8}

The following proof follows closely the argument of R. R{\u a}sdeaconu and I. {\c S}uvaina used to prove Theorem 1.1 in \cite{[RS]}.

\begin{proof} Theorem 3 settles the existence part of the proposition. According to \cite{[PPS]}, in $Z$ there are five disjoint linear chains. Using the labels we put on their dual graphs in Section 2, let us denote them by $G  = \sum^{11}_{i = 1} G_i$, $H  = \sum^{7}_{i = 1} H_i$, $I  = \sum^{4}_{i = 1} I_i$, $J  = \sum^{3}_{i = 1} J_i$, and let $L$ be the chain of length one. Name $F_i$, $i = 1, \ldots, 11$ the eleven smooth curves of self-intersection -1 represented by dotted lines labeled -1 in Fig 13 of \cite{[PPS]}. We point out that the Poincar\'e duals of the irreducible components of the five chains and those of the curves $F_i$'s form a basis of $H^2(Z; \Q)$.

Let $f: Z\rightarrow X$ be the contraction map. Then, one has

\begin{center}
$f^* K_X \equiv_{\Q} \displaystyle\sum_{i=1}^{11} a_i F_i + \displaystyle\sum_{i=1}^{10} b_i G_i + \displaystyle\sum_{i=1}^{7} c_i H_i + \displaystyle\sum_{i=1}^{4} d_i I_i + \displaystyle\sum_{i=1}^{3} e_i J_i + l_1 L$.
\end{center}

The coefficients that appear above can be computed explicitly (see \cite{[PPS2]}). However, for our agenda it suffices to know that they are positive rational numbers. In particular the pullback of the canonical divisor of the singular variety to its minimal resolution is effective. Set the exceptional divisor of f to be $Exc(f) = \sum G_i + \sum H_i + \sum I_i + \sum J_i + L$.


We wish to show that the canonical bundle $K_X$ of the $\Q$-Gorenstein smoothing is ample. This implies our claim: indeed, remember $S$ is a general fiber of the $\Q$-Gorenstein smoothing $X$, and ampleness is an open property (\cite{[KM]}). Moreover, we know $K_X$ is nef. To show it is ample as well, we proceed by contradiction.\\

Suppose $K_X$ is not ample. By its nefness and according to the Nakai-Moishezon criterion (\cite{[KM]}), there exists an irreducible curve $C \subset X$ such that $(K_X \cdot C) = 0$.\\

The total transform of $C$ in $Z$ is

\begin{center}

$f^* C \equiv_{\Q} C' + \displaystyle\sum_{i=1}^{10} w_i G_i + \displaystyle\sum_{i=1}^{7} x_i H_i + \displaystyle\sum_{i=1}^{4} y_i I_i + \displaystyle\sum_{i=1}^{3} z_i J_i + t L$.

\end{center}

Here $C'$ stands for the strict transform of $C$, and the coefficientes $w_i, x_i, y_i, z_i, t$ are non-negative rational numbers. It is straight-forward to see that $C'$ is not numerically equivalent to $0$ (\cite{[RS]}).\\

We compute

\begin{center}
$(K_X \cdot C) = (f^* K_X \cdot f^* C) = (f^* \cdot C') =$\\  $= \displaystyle\sum_{i=1}^{11} a_i (F_i \cdot C') + \displaystyle\sum_{i=1}^{10} b_i(G_i \cdot C') + \displaystyle\sum_{i=1}^{7} c_i (H_i \cdot C') +$\\ $ + \displaystyle\sum_{i=1}^{4} d_i (I_i\cdot C') + \displaystyle\sum_{i=1}^{3} e_i (J_i \cdot C')+ l_1 (L\cdot C')$.\\
\end{center}

The intersection number of the curve $C'$ with any component of the exceptional divisor $Exc(f)$ is greater or equal to zero. The equality is achieved only in the case when $C'$ is disjoint to all the irreducible components of $Exc(f)$; this is equivalent to the curve $C$ missing the singular points of $X$. This is

\begin{center}
$\displaystyle\sum_{i=1}^{10} b_i(G_i \cdot C') + \displaystyle\sum_{i=1}^{7} c_i (H_i \cdot C')  + \displaystyle\sum_{i=1}^{4} d_i (I_i\cdot C') + \displaystyle\sum_{i=1}^{3} e_i (J_i \cdot C')+ l_1 (L\cdot C') \geq 0$.
\end{center}

Thus, we have $\displaystyle\sum_{i=1}^{11} a_i (F_i \cdot C') \leq 0$. At this point there are two possible scenarios.
\begin{itemize}
\item Either there is an $i_0 \in \{1, \ldots , 11\}$ such that $(C'\cdot F_{i_0} ) < 0$, or
\item the equality $(C' \cdot F_i) = 0$ holds for all $i = 1, \ldots, 11$.
\end{itemize}

The first scenario requires $C'$ to coincide with $F_{i_0}$. This is not the case, since given that $f^*K_X$ is nef, $(f^* K_X \cdot F_i) >$ holds for all $i = 1, \ldots, 11$, which is impossible by our assumption. Thus, the intersection number of the curve $C'$ with all the $F_i$'s and with all of the irreducible components of $Exc(f)$ must be zero. However, as it was remarked earlier, the Poincar\'e duals of the $F_i$'s and those of the irreducible components of $Exc(f)$ generate $H^2(Z; \Q)$. This implies that $C'$ would have to be numerically trivial on $Z$. This is a contradiction.\\

Thus, $K_X$ is ample. The proposition now follows from Aubin-Yau's criterion (Theorem 7).
\end{proof}

\begin{corollary} There exist a minimal complex structure  on $3\mathbb{CP}^2 \# k\overline{\mathbb{CP}^2}$, for each $k = 11, 13, 14, 15, 16, 17, 18$, which admits a K\"ahler-Einstein metric of negative scalar curvature.
\end{corollary}

\section{Non-existence of Einstein metrics: Exotic smooth structures}

Topologically there is no obstruction for the existence of an Einstein metric on the surfaces of general type we are working with (cf. Corollary 5). We now proceed to study the non-existence of Einstein metrics with respect to their exotic differential structures.\\

When one considers different smooth structures on 4-manifolds, the main obstruction to the existence of an Einstein metric is the following result, which generalizes work done by C. LeBrun in \cite{[LB]} and by D. Kotschick \cite{[DK]}.

\begin{theorem}{\label{Theorem C}} (LeBrun, \cite{[LB4]}). Let $X$ be a compact oriented 4-manifold with a non-trivial Seiberg-Witten invariant and with $(2\chi + 3\sigma)(X) > 0$. Then
\begin{center}
$M = X \# r \overline{\mathbb{CP}^2}$
\end{center}
does not admit an Einstein metric if $r \geq \frac{1}{3}(2\chi + 3\sigma)(X)$.
\end{theorem}

As a corollary we have.

\begin{proposition} Let $9\leq k\leq 18$. The topological manifolds
\begin{center}
$3\mathbb{CP}^2 \# k \overline{\mathbb{CP}^2}$
\end{center}
support infinitely many smooth structures that do not admit an Einstein metric. Moreover, each of these manifolds admits infinitely many smooth structures, all of which have negative Yamabe invariant, and on which there are no non-singular solutions to the normalized Ricci flow for any initial metric.
\end{proposition}

\begin{proof} We make use of the infinite family $\{X_n\}$ of pairwise non-diffeomorphic 4-manifolds (with non-trivial SW) sharing the topological prototype
$3\mathbb{CP}^2 \# 4\overline{\mathbb{CP}^2}$ built in \cite{[AP]}. The first part of the lemma now follows by setting $r\geq 5$ in LeBrun's result (Theorem \ref{Theorem C}); notice that the blow-up formula \cite[Theorem 1.4]{[FS]} allows us to conclude that the manifolds in the infinite family $\{X_n \# (4 + r)\overline{\mathbb{CP}^2}\}$ are pairwise non-diffeomorphic. For the claims regarding the Yamabe invariant and the solutions to the normalized Ricci flow see Section 5 below.
\end{proof}


\subsection{Non-existence of anti-self-dual Einstein metrics} Using another obstruction theorem of LeBrun in \cite{[LB4]} we obtain the following lemma.

\begin{lemma}Let $7\leq k\leq 18$. The topological manifolds
\begin{center}
$3\mathbb{CP}^2 \# k \overline{\mathbb{CP}^2}$
\end{center}
support infinitely many smooth structures that do not admit an anti-self-dual Einstein metric.
\end{lemma}

\section{Proof of Proposition 2}

The following argument is based on the proof of Theorem B in \cite{[IRS]}.

\begin{proof} We start with the part of (1) concerning the sign of the Yamabe invariant. Consider the smooth structure related to the minimal surfaces of general type taken from \cite{[PPS]}. By \cite{[LB2]} their Yamabe invariant is negative. Regarding the existence of non-singular solutions to the normalized Ricci flow, it follows from Cao's theorem (\cite{[C1]}, \cite{[C2]}) by taking as an initial metric the K\"ahler metric with K\"ahler form the cohomology class of the canonical line bundle.\\

For Property (2), consider the smooth structures used in Theorem 1 that were built by A. Akhmedov and B.D. Park in \cite{[AP]}: the infinite family $\{X_n\}$ of pairwise non-diffeomorphic minimal manifolds homeomorphic to $3\mathbb{CP}^2 \# 4 \overline{\mathbb{CP}^2}$ . These manifolds have non-trivial Seiberg-Witten invariants, and for all of them $c_1^2 > 0$ holds. Thus, by \cite{[LB]}, their Yamabe invariant is strictly negative. By a result of M. Ishida (Theorem B in \cite{[MI]}), there are no solutions to the normalized Ricci flow on $X_i$ for any $i$ and any intial metric.
\end{proof}

\end{document}